\documentclass[a4paper]{article}

\makeatletter
\providecommand*{\input@path}{}
\g@addto@macro\input@path{{}}
\makeatother

\usepackage{graphicx}       

\usepackage[latin1]{inputenc}
\usepackage[T1]{fontenc}
\usepackage{microtype} 

\usepackage{cite}
\usepackage{array,colortbl}
\usepackage{amsmath,amsfonts,amssymb,bm,amsthm,mathtools,mathrsfs}
\DeclareFontFamily{U}{mathx}{\hyphenchar\font45}
\DeclareFontShape{U}{mathx}{m}{n}{
      <5> <6> <7> <8> <9> <10>
      <10.95> <12> <14.4> <17.28> <20.74> <24.88>
      mathx10
      }{}
\DeclareSymbolFont{mathx}{U}{mathx}{m}{n}
\DeclareFontSubstitution{U}{mathx}{m}{n}
\DeclareMathAccent{\widecheck}{0}{mathx}{"71}

\def\citep#1#2{\cite[{#1}]{#2}}

\newcommand{\bbE}{{\mathbb{E}}}

\newcommand{\bbN}{{\mathbb{N}}}

\newcommand{\N}{{\mathbb{N}}} 
\newcommand{\R}{{\mathbb{R}}} 
\newcommand{\Z}{{\mathbb{Z}}} 
\newcommand{\NN}{{\mathbb{N}}} 
\newcommand{\RR}{{\mathbb{R}}} 
\newcommand{\ZZ}{{\mathbb{Z}}} 
\newcommand{\EE}{{\mathbb{E}}}

\DeclareSymbolFont{bbold}{U}{bbold}{m}{n}
\DeclareSymbolFontAlphabet{\mathbbold}{bbold}

\newcommand{\bsb}{{\boldsymbol{b}}}

\newcommand{\bse}{{\boldsymbol{e}}}
\newcommand{\bsf}{{\boldsymbol{f}}}
\newcommand{\bsg}{{\boldsymbol{g}}}

\newcommand{\bsq}{{\boldsymbol{q}}}

\newcommand{\bst}{{\boldsymbol{t}}}
\newcommand{\bsu}{{\boldsymbol{u}}}
\newcommand{\bsv}{{\boldsymbol{v}}}
\newcommand{\bsw}{{\boldsymbol{w}}}
\newcommand{\bsx}{{\boldsymbol{x}}}
\newcommand{\bsy}{{\boldsymbol{y}}}
\newcommand{\bsz}{{\boldsymbol{z}}}

\newcommand{\bsT}{{\boldsymbol{T}}}

\newcommand{\bszero}{{\boldsymbol{0}}} 

\newcommand{\bsgamma}{{\boldsymbol{\gamma}}}

\newcommand{\bsnu}{{\boldsymbol{\nu}}}

\newcommand{\bsomega}{{\boldsymbol{\omega}}}

\newcommand{\bsDelta}{{\boldsymbol{\Delta}}}

\newcommand{\calF}{{\mathcal{F}}}

\newcommand{\calH}{{\mathcal{H}}}

\newcommand{\calO}{{\mathcal{O}}}

\newcommand{\calW}{{\mathcal{W}}}

\newcommand{\fraku}{{\mathfrak{u}}}

\newcommand{\setu}{{\mathfrak{u}}}

\newcommand{\rmd}{{\mathrm{d}}}
\newcommand{\rme}{{\mathrm{e}}}

\newcommand{\rmi}{{\mathrm{i}}}

\newcommand{\abs}[1]{\left\vert#1\right\vert}

\newcommand{\rd}{\,\mathrm{d}} 

\providecommand{\argmin}{\operatorname*{argmin}}

\usepackage[hidelinks,hypertexnames=false,hyperindex=true,pdfpagelabels,linktoc=all]{hyperref}
\usepackage{bookmark}
\makeatletter 
  \providecommand*{\toclevel@author}{999}
  \providecommand*{\toclevel@title}{0}
\makeatother

\usepackage{enumitem}

\theoremstyle{plain}
  \newtheorem{theorem}{Theorem}

  \newtheorem{corollary}{Corollary}
  
\theoremstyle{definition}

\theoremstyle{remark}

\usepackage[a4paper]{geometry}

\usepackage{algorithm}
\usepackage{algorithmic}
\usepackage{setspace}

\allowdisplaybreaks

\newcommand{\supp}{{\mathrm{supp}}}
\DeclareMathOperator*{\essinf}{ess\,inf}

\numberwithin{equation}{section} 

\usepackage{xcolor}

\begin{document}

\title{Constructing QMC Finite Element Methods \\ for Elliptic PDEs with Random Coefficients 
	    \\ by a Reduced CBC Construction}
\author{Adrian Ebert \and Peter Kritzer\thanks{Supported by the Austrian Science Fund (FWF) Project  F5506-N26, 
part of the Special Research Program ``Quasi-Monte Carlo Methods: Theory and Applications''.} \and Dirk Nuyens}

\maketitle

\abstract{
In the analysis of using quasi-Monte Carlo (QMC) methods to approximate expectations of a linear functional of the solution of an elliptic PDE with random diffusion coefficient the sensitivity w.r.t.\ the parameters is often stated in terms of product-and-order-dependent (POD) weights. The (offline) fast component-by-component (CBC) construction of an $N$-point QMC method making use of these POD weights leads to a cost of $\calO(s N\log(N) + s^2 N)$ with $s$ the parameter truncation dimension. When $s$ is large this cost is prohibitive. As an alternative Herrmann and Schwab \cite{HS18} introduced an analysis resulting in product weights to reduce the construction cost to $\calO(s N \log(N))$. We here show how the reduced CBC method can be used for POD weights to reduce the cost to $\mathcal{O}(\sum_{j=1}^{\min\{s,s^{\ast}\}} (m-w_j+j) \, b^{m-w_j})$, where $N=b^m$ with prime $b$, $w_1 \le \cdots \le w_s$ are nonnegative integers and $s^*$ can be chosen much smaller than $s$ depending on the regularity of the random field expansion as such making it possible to use the POD weights directly. We show a total error estimate for using randomly shifted lattice rules constructed through the reduced CBC construction.}

\section{Introduction and Problem Setting}

We consider the parametric elliptic Dirichlet problem given by
\begin{align} \label{eq:PDE}
	- \nabla \cdot (a(\bsx,\bsy)\,\nabla u(\bsx,\bsy)) 
	&= 
	f(\bsx) \;\;\; \mbox{for} \,\, \bsx \in D \subset \R^d,
	\;\;\; u(\bsx,\bsy) = 0 \;\;\; \mbox{for} \,\, \bsx \,\, \mbox{on} \,\, \partial D,
\end{align}
for $D \subset \R^d$ a bounded, convex Lipschitz polyhedron domain with boundary $\partial D$ and fixed spatial dimension 
$d \in \{1,2,3\}$. The function $f$ lies in $L^2(D)$, the parametric variable $\bsy = (y_j)_{j\ge1}$ belongs to a domain $U$, 
and the differential operators are understood to be with respect to the physical variable $\bsx \in D$. Here we study the ``uniform case'', 
i.e., we assume that $\bsy$ is uniformly distributed on $U :=\left[-\frac12,\frac12\right]^{\N}$ with uniform probability measure 
$\mu(\rmd\bsy) = \bigotimes_{j\geq 1} \rmd y_j = \rmd\bsy$.
The parametric diffusion coefficient $a(\bsx,\bsy)$ is assumed to depend linearly on the parameters $y_j$ in the following way,
\begin{equation} \label{eq:diff_coeff}
	a(\bsx,\bsy)
	= 
	a_0(\bsx) + \sum_{j\geq 1} y_j\, \psi_j(\bsx)\,, \quad \bsx \in D, \quad \bsy \in U.
\end{equation}
For the variational formulation of \eqref{eq:PDE}, we consider the Sobolev space $V = H_0^1(D) $ of functions $v$
which vanish on the boundary $\partial D$ with norm
\begin{equation*}
	\|v\|_V := \left( \int_D \sum_{j=1}^{d} |\partial_{x_j} v(\bsx)|^2 \rd\bsx \right)^{\frac12} = \|\nabla v\|_{L^2(D)}.
\end{equation*}
The corresponding dual space of bounded linear functionals on $V$ with respect to the pivot space $L^2(D)$ is
further denoted by $V^* = H^{-1}(D)$. Then, for given $f\in V^*$ and $\bsy\in U$, the weak (or variational) formulation of~\eqref{eq:PDE} 
is to find $u(\cdot,\bsy)\in V$ such that
\begin{equation} \label{eq:PDE_weak}
	A(\bsy;u(\cdot,\bsy), v) 
	= 
	\langle f,v\rangle_{V^*\times V} = \int_D f(\bsx) v(\bsx)\,\rd\bsx \quad\mbox{for all}\quad v \in V,
\end{equation}
with parametric bilinear form $A: U \times V \times V \to \R$ given by
\begin{equation} \label{eq:bilinear_form}
	A(\bsy; w,v) 
	:=
	\int_D a(\bsx,\bsy)\,\nabla w(\bsx)\cdot\nabla v(\bsx)\,\rd\bsx \quad\mbox{for all}\quad w, v\in V,
\end{equation}
and duality pairing $\langle\cdot,\cdot\rangle_{V^* \times V}$ between $V^*$ and $V$. We will often 
identify elements $\varphi \in V$ with dual elements $L_\varphi \in V^*$. Indeed, for $\varphi \in V$ and $v\in V$, a bounded linear 
functional is given via $L_\varphi(v) := \int_D \varphi(\bsx) v(\bsx) \rd\bsx = \langle \varphi,v \rangle_{L^2(D)}$ and by the Riesz representation 
theorem there exists a unique representer $\widetilde{\varphi} \in V$ such that $L_\varphi(v) = \langle \widetilde\varphi,v \rangle_{L^2(D)}$
for all $v\in V$. Hence,  the definition of the canonical duality pairing yields that
$\langle L_\varphi,v \rangle_{V^* \times V} = L_\varphi(v) = \langle \varphi,v \rangle_{L^2(D)}$.

Our quantity of interest is the expected value, with respect to $\bsy\in U$, of a given bounded linear 
functional $G \in V^*$ applied to the solution $u(\cdot,\bsy)$ of the PDE.
We therefore seek to approximate this expectation by numerically integrating $G$ applied 
to a finite element approximation $u_h^s(\cdot,\bsy)$ of the solution $u^s(\cdot,\bsy) \in H_0^1(D) = V$
of \eqref{eq:PDE_weak} with truncated diffusion coefficient $a(\bsx,(\bsy_{\{1:s\}};0))$ where $\{1:s\}:=\{1,\ldots,s\}$ and we write 
$(\bsy_{\{1:s\}};0) = (\tilde{y}_j)_{j \ge 1}$ with $\tilde{y}_j = y_j$ for $j \in \{1:s\}$ and $\tilde{y}_j = 0$ otherwise; that is,
\begin{equation} \label{eq:QoI}
	\EE[G(u)]
	:=
	\int_{U} G(u(\cdot,\bsy))\,\mu(\rmd\bsy)
	=
	\int_{U} G(u(\cdot,\bsy))\,\rd\bsy
	\approx
	Q_N(G(u_h^s)) ,
\end{equation}
with $Q_N(\cdot)$ a linear quadrature rule using $N$ function evaluations.
The infinite-dimensional integral $\EE[G(u)]$ in~\eqref{eq:QoI} is defined as
\begin{equation*}
	\EE[G(u)]
	=
	\int_{U} G(u(\cdot,\bsy))\,\rd\bsy
	:= 
	\lim_{s\to\infty} \int_{\left[-\frac12,\frac12\right]^s} G(u(\cdot,(y_1,\ldots,y_s,0,0,\ldots)))\,\rd y_1\cdots\rd y_s
\end{equation*}
such that our integrands of interest are of the form $F(\bsy) = G(u(\cdot,\bsy))$ with $\bsy \in U$. In this article, 
we will employ (randomized) QMC methods of the form
\begin{equation*}
	Q_N(f) = \frac 1N \sum_{k=1}^{N} F(\bst_k),
\end{equation*}
i.e., equal-weight quadrature rules with (randomly shifted) deterministic points $\bst_1,\ldots,\bst_{N} \in \left[-\frac12,\frac12\right]^s$. This elliptic PDE is a standard problem considered in the numerical analysis of computational methods in uncertainty quantification, see, e.g., \cite{BCM17,CDS06,DKLNS14,GHS18,HS18,K17,KN16,KSS12}.

\subsection{Existence of solutions of the variational problem}
\label{subsec:PDE}

To assure that a unique solution to the weak problem \eqref{eq:PDE_weak} exists, we need certain conditions 
on the diffusion coefficient $a$. We assume 
$a_0 \in L^{\infty}(D)$ and $\essinf_{\bsx \in D} a_0(\bsx) > 0$,
which is equivalent to the existence of two constants $0 < a_{0,\min} \le a_{0,\max} < \infty$ such that
a.e.\ on $D$ we have
\begin{equation} \label{eq:bounds_a_1}
	a_{0,\min} \le a_0(\bsx) \le a_{0,\max} ,
\end{equation}
and that there exists a $\overline{\kappa} \in (0,1)$ such that
\begin{equation} \label{eq:kappa_bar}
	\left\| \sum_{j \ge 1} \frac{|\psi_j|}{2 a_0} \right\|_{L^{\infty}(D)}
	\le
	\overline{\kappa}
	<
	1 .
\end{equation}
Via \eqref{eq:kappa_bar}, we obtain that $|\sum_{j \ge 1} y_j \psi_j(\bsx)| \le \overline{\kappa} \, a_0(\bsx)$ and 
hence, using \eqref{eq:bounds_a_1}, almost everywhere on $D$ and for any $\bsy \in U$
\begin{align} \label{eq:bound_a_2}
	0 < (1-\overline{\kappa}) \, a_{0,\min} \le a_0(\bsx) + \sum_{j \ge 1} y_j \psi_j(\bsx) = a(\bsx,\bsy) \le (1+\overline{\kappa}) \, a_{0,\max} .
\end{align}
These estimates yield the continuity and coercivity of $A(\bsy,\cdot,\cdot)$ defined in \eqref{eq:bilinear_form}
on $V \times V$, uniformly for all $\bsy \in U$. The Lax--Milgram theorem then ensures the existence of a unique solution 
$u(\cdot,\bsy)$ of the weak problem in \eqref{eq:PDE_weak}.

\subsection{Parametric regularity}
\label{subsec:PDE}

Having established the existence of unique weak parametric solutions $u(\cdot,\bsy)$, we investigate their regularity in terms of the 
behaviour of their mixed first-order derivatives. 
Our analysis combines multiple techniques which can be found in the literature, see, e.g., \cite{CDS06,GHS18,K17,HS18,BCM17}. In particular we want to point out that our POD form bounds can take advantage of wavelet like expansions of the random field, a technique introduced in \cite{BCM17} and used to the advantage of QMC constructions by \cite{HS18} to deliver product weights to save on the construction compared to POD weights. Although we end up again with POD weights, we will save on the construction cost by making use of a special construction method, called the reduced CBC construction, which we will introduce in~Section \ref{sec:fast-reduced-CBC-POD}.
Let $\bsnu = (\nu_j)_{j \ge 1}$ with $\nu_j \in \N_0 := \{0,1,2,\ldots\}$ be a sequence of positive integers which we will refer to as a multi-index. We define the order $|\bsnu|$ and the support $\supp(\bsnu)$ as
\begin{equation*}
	|\bsnu| := \sum_{j \ge 1} \nu_j
	\quad \text{and} \quad
	\supp(\bsnu) := \{j \ge 1 : \nu_j > 0 \}
\end{equation*}
and introduce the sets $\calF$ and $\calF_1$ of finitely supported multi-indices as
\begin{equation*}
	\calF
	:=
	\{ \bsnu \in \NN_0^\bbN : \supp(\bsnu) < \infty \}
	\quad \text{and} \quad
	\calF_1
	:=
	\{ \bsnu \in \{0,1\}^\bbN : \supp(\bsnu) < \infty \} ,
\end{equation*}
where $\calF_1 \subseteq \calF$ is the restriction containing only $\bsnu$ with $\nu_j \in \{0,1\}$.
Then, for $\bsnu \in \calF$ denote the $\bsnu$-th partial derivative with respect to the parametric variables $\bsy \in U$ by
\begin{equation*}
	\partial^{\bsnu}
	=
	\frac{\partial^{|\bsnu|}}{\partial y_1^{\nu_1}\partial y_2^{\nu_2}\cdots},
\end{equation*}
and for a sequence $\bsb = (b_j)_{j\ge 1} \subset \R^{\N}$, set $\bsb^\bsnu := \prod_{j\ge 1} b_j^{\nu_j}$.
We further write $\bsomega \le \bsnu$ if $\omega_j \le \nu_j$ for all $j \ge 1$ and denote by $\bse_i \in \calF_1$ the multi-index
with components $e_j = \delta_{i,j}$.  
For a fixed $\bsy \in U$, we introduce the energy norm $\|\cdot\|_{a_\bsy}^2$ in the space $V$ via
\begin{equation*}
	\|v\|_{a_\bsy}^2
	:=
	\int_{D} a(\bsx,\bsy) \, |\nabla v(\bsx)|^2 \, \rd \bsx
\end{equation*}
for which it holds true by \eqref{eq:bound_a_2} that 
\begin{equation} \label{eq:connection_norms}
	(1-\overline{\kappa}) \, a_{0,\min} \|v\|_V^2 \le \|v\|_{a_\bsy}^2 \quad \text{for all} \quad v \in V .
\end{equation}
Consequently, we have that $(1-\overline{\kappa}) \, a_{0,\min} \|u(\cdot,\bsy)\|_V^2 \le \|u(\cdot,\bsy)\|_{a_\bsy}^2$
and hence the definition of the dual norm $\|\cdot\|_{V^*}$ yields the following initial estimate from~\eqref{eq:PDE_weak} and~\eqref{eq:bilinear_form},
\begin{align*}
	\|u(\cdot,\bsy)\|_{a_\bsy}^2 
	&= 
	\int_{D} a(\bsx,\bsy) \, |\nabla u(\bsx,\bsy)|^2 \, \rd \bsx
	=
	\int_{D} f(\bsx) u(\bsx,\bsy) \, \rd \bsx \\
	&=
	\langle f,u(\cdot,\bsy) \rangle_{V^* \times V}
	\le
	\|f\|_{V^*} \|u(\cdot,\bsy)\|_V
	\le 
	\frac{\|f\|_{V^*} \|u(\cdot,\bsy)\|_{a_\bsy}}{\sqrt{(1-\overline{\kappa}) a_{0,\min}}}
\end{align*}
which gives in turn
\begin{equation} \label{est:norm_u_a}
	\|u(\cdot,\bsy)\|_{a_\bsy}^2 
	\le 
	\frac{\|f\|_{V^*}^2 }{(1-\overline{\kappa}) \, a_{0,\min}} .
\end{equation}
In order to exploit the decay of the norm sequence $(\|\psi_j\|_{L^\infty(D)})_{j \ge 1}$ of the basis functions, we extend
condition \eqref{eq:kappa_bar} as follows. To characterize the smoothness of the random field, we assume that there exist a sequence of reals $\bsb = (b_j)_{j\ge1}$ with $0 < b_j \le 1$ for all $j$, a constant $\kappa \in (0,1)$ and therefore also constants $\widetilde{\kappa}(\bsnu) \le \kappa$ for all $\bsnu \in \calF_1$ such that
\begin{align} \label{def:kappa}
	\kappa
	&:=
	\left\| \sum_{j \ge 1} \frac{|\psi_j| /b_j }{2 a_0} \right\|_{L^{\infty}(D)} 
	< 
	1
	,
	&
	\widetilde{\kappa}(\bsnu) = \left\| \sum_{j \in \supp(\bsnu)} \frac{|\psi_j| / b_j}{2a_0} \right\|_{L^{\infty}(D)}
	.
\end{align}
We remark that condition \eqref{eq:kappa_bar} is included in this assumption by letting $b_j = 1$ for all $j \ge 1$
and that $0 < \overline{\kappa} \le \kappa < 1$. Using the above estimations we can derive the following theorem for 
the mixed first-order partial derivatives.
\begin{theorem} \label{thm:deriv_bound}
	Let $\bsnu \in \calF_1$ be a multi-index of finite support and let $k \in \{0,1,\ldots,|\bsnu|\}$.
	Then, for every $f\in V^*$ and every $\bsy\in U$, 
	\begin{equation*}
		\sum_{\substack{\bsomega \le \bsnu \\ |\bsomega| = k}} \bsb^{-2\bsomega} \|\partial^{\bsomega}u(\cdot,\bsy)\|_{V}^2
		\le
		\left( \left(\frac{2 \widetilde{\kappa}(\bsnu)}{1-\overline{\kappa}}\right)^{k} \frac{\|f\|_{V^*} }{(1-\overline{\kappa}) \, a_{0,\min}} \right)^2,
	\end{equation*}
	with $\widetilde{\kappa}(\bsnu)$ as in~\eqref{def:kappa}. Moreover, for $k = |\bsnu|$ we obtain
	\begin{equation*}
		\|\partial^{\bsnu} u(\cdot,\bsy) \|_{V}
		\le
		\bsb^{\bsnu} \left(\frac{2\widetilde{\kappa}(\bsnu)}{1-\overline{\kappa}}\right)^{|\bsnu|} \frac{\|f\|_{V^*}}{(1-\overline{\kappa}) \, a_{0,\min}} .
	\end{equation*}
\end{theorem}

\begin{proof}
	For the special case $\bsnu = \bszero$, the claim follows by combining \eqref{eq:connection_norms} and \eqref{est:norm_u_a}.
	For $\bsnu \in \calF_1$ with $\abs{\bsnu} > 0$, as is known from, e.g., \cite{CDS06} and \cite[Appendix]{KN16}, the linearity of $a(\bsx,\bsy)$   
	gives rise to the following identity for any $\bsy \in U$:
	\begin{equation} \label{eq:deriv_leibniz_form}
		\|\partial^{\bsnu} u(\cdot,\bsy) \|_{a_\bsy}^2 
		=
		- \sum_{j \in \supp(\bsnu)} \int_{D} \psi_j(\bsx) \, \nabla \partial^{\bsnu - \bse_j} u(\bsx,\bsy) 
		\cdot \nabla \partial^{\bsnu} u(\bsx,\bsy) \, \rd \bsx .
	\end{equation}
	For sequences of $L^2(D)$-integrable functions $\bsf = (f_{\bsomega,j})_{\bsomega \in \calF, j \ge 1}$ with $f_{\bsomega,j}: D \to \R$, we define the inner product $\langle \bsf, \bsg \rangle_{\bsnu,k}$ as follows,
	\begin{equation*}
		\langle \bsf, \bsg \rangle_{\bsnu,k}
		:=
		\sum_{\substack{\bsomega \le \bsnu \\ |\bsomega| = k}} \int_{D} \sum_{j \in \supp(\bsomega)}
		f_{\bsomega,j}(\bsx) \, g_{\bsomega,j}(\bsx) \, \rd \bsx .
	\end{equation*}
	We can then apply the Cauchy--Schwarz inequality to $\bsf = (f_{\bsomega,j})$ and $\bsg = (g_{\bsomega,j})$ with 
	$f_{\bsomega,j} = \bsb^{-\bse_j/2} |\psi_j|^{\frac12} \,  \bsb^{-(\bsomega-\bse_j)} \nabla \partial^{\bsomega - \bse_j} u(\cdot,\bsy)$ and $g_{\bsomega,j} = \bsb^{-\bse_j/2} |\psi_j|^{\frac12} \,  \bsb^{-\bsomega} \nabla \partial^{\bsomega} u(\cdot,\bsy)$ to obtain,
	with the help of \eqref{eq:deriv_leibniz_form},
	\begin{align*}
		&\sum_{\substack{\bsomega \le \bsnu \\ |\bsomega| = k}} \bsb^{-2\bsomega} \|\partial^{\bsomega}u(\cdot,\bsy)\|_{a_{\bsy}}^2 \\
		&\quad=
		-\sum_{\substack{\bsomega \le \bsnu \\ |\bsomega| = k}} \int_{D} \sum_{j \in \supp(\bsomega)} \bsb^{-\bse_j} \bsb^{-(\bsomega-\bse_j)} \bsb^{-\bsomega} \psi_j(\bsx) \, 
		\nabla \partial^{\bsomega - \bse_j} u(\bsx,\bsy) \cdot \nabla \partial^{\bsomega} u(\bsx,\bsy) \, \rd \bsx \\
		&\quad\le
		\left( \int_{D} \sum_{\substack{\bsomega \le \bsnu \\ |\bsomega| = k}}  \sum_{j \in \supp(\bsomega)} \bsb^{-\bse_j} |\psi_j(\bsx)| \, 
		\left| \bsb^{-(\bsomega-\bse_j)} \nabla \partial^{\bsomega - \bse_j} u(\bsx,\bsy) \right|^2 \rd \bsx \right)^{\frac12} \\
		&\quad\qquad 
		\times\left( \int_{D} \sum_{\substack{\bsomega \le \bsnu \\ |\bsomega| = k}}  \sum_{j \in \supp(\bsomega)} \bsb^{-\bse_j} |\psi_j(\bsx)| \, 
		\left| \bsb^{-\bsomega} \nabla \partial^{\bsomega} u(\bsx,\bsy) \right|^2 \rd \bsx \right)^{\frac12} .
	\end{align*}
	The first of the two factors above is then bounded as follows,
	\begin{align*}
		&\int_{D} \sum_{\substack{\bsomega \le \bsnu \\ |\bsomega| = k}}  \sum_{j \in \supp(\bsomega)} \bsb^{-\bse_j} 
		|\psi_j(\bsx)| \, 
		\left| \bsb^{-(\bsomega-\bse_j)} \nabla \partial^{\bsomega - \bse_j} u(\bsx,\bsy) \right|^2 \rd \bsx \\
		&\quad=
		\int_{D} \sum_{\substack{\bsomega \le \bsnu \\ |\bsomega| = k-1}} \Bigg(\sum_{\substack{j \in \supp(\bsnu) \\ \bsomega + \bse_j \le \bsnu}} 
		\bsb^{-\bse_j} |\psi_j(\bsx)| \Bigg) \left| \bsb^{-\bsomega} \nabla \partial^{\bsomega} u(\bsx,\bsy) \right|^2 \rd \bsx \\
		&\quad\le
		\left\| \sum_{j \in \supp(\bsnu)} \frac{|\psi_j| / b_j}{a(\cdot,\bsy)} \right\|_{L^{\infty}(D)} 
		\sum_{\substack{\bsomega \le \bsnu \\ |\bsomega| = k-1}} \bsb^{-2\bsomega} \int_{D} a(\bsx,\bsy) \left| \nabla \partial^{\bsomega} u(\bsx,\bsy) \right|^2 \rd \bsx \\
		&\quad=
		\left\| \sum_{j \in \supp(\bsnu)} \frac{|\psi_j| / b_j}{a(\cdot,\bsy)} \right\|_{L^{\infty}(D)} 
		\sum_{\substack{\bsomega \le \bsnu \\ |\bsomega| = k-1}} \bsb^{-2\bsomega} \|\partial^{\bsomega}u(\cdot,\bsy)\|_{a_{\bsy}}^2 ,
	\end{align*}
	while the other factor can be bounded trivially. Furthermore, using \eqref{eq:bound_a_2}, we have for any $\bsy \in U$
	\begin{equation*}
		\left\| \sum_{j \in \supp(\bsnu)} \frac{|\psi_j| / b_j}{a(\cdot,\bsy)} \right\|_{L^{\infty}(D)}
		\le
		\frac{1}{1-\overline{\kappa}} \left\| \sum_{j \in \supp(\bsnu)} \frac{|\psi_j| / b_j}{a_0} \right\|_{L^{\infty}(D)} := \frac{2 \widetilde{\kappa}(\bsnu)}{1-\overline{\kappa}},
	\end{equation*}
	so that, combining these three estimates, we obtain
	\begin{align*}
		&\sum_{\substack{\bsomega \le \bsnu \\ |\bsomega| = k}} \bsb^{-2\bsomega} \|\partial^{\bsomega}u(\cdot,\bsy)\|_{a_{\bsy}}^2 \\
		&\phantom{=}\le
		\frac{2 \widetilde{\kappa}(\bsnu)}{1-\overline{\kappa}}
		\left( \sum_{\substack{\bsomega \le \bsnu \\ |\bsomega| = k-1}} \bsb^{-2\bsomega} \|\partial^{\bsomega}u(\cdot,\bsy)\|_{a_{\bsy}}^2 \right)^{\frac12}
		\left( \sum_{\substack{\bsomega \le \bsnu \\ |\bsomega| = k}} \bsb^{-2\bsomega} \|\partial^{\bsomega}u(\cdot,\bsy)\|_{a_{\bsy}}^2 \right)^{\frac12} .
	\end{align*}
	Therefore, we finally obtain that 
	\begin{equation*} 
		\sum_{\substack{\bsomega \le \bsnu \\ |\bsomega| = k}} \bsb^{-2\bsomega} \|\partial^{\bsomega}u(\cdot,\bsy)\|_{a_{\bsy}}^2
		\le
		\left(\frac{2 \widetilde{\kappa}(\bsnu)}{1-\overline{\kappa}}\right)^2
		\sum_{\substack{\bsomega \le \bsnu \\ |\bsomega| = k-1}} \bsb^{-2\bsomega} \|\partial^{\bsomega}u(\cdot,\bsy)\|_{a_{\bsy}}^2
	\end{equation*}
	which inductively gives
	\begin{equation*}
		\sum_{\substack{\bsomega \le \bsnu \\ |\bsomega| = k}} \bsb^{-2\bsomega} \|\partial^{\bsomega}u(\cdot,\bsy)\|_{a_{\bsy}}^2
		\le
		\left(\frac{2\widetilde{\kappa}(\bsnu)}{1-\overline{\kappa}}\right)^{2k} \|u(\cdot,\bsy)\|_{a_{\bsy}}^2
		\le
		\left(\frac{2\widetilde{\kappa}(\bsnu)}{1-\overline{\kappa}}\right)^{2k} \frac{\|f\|_{V^*}^2 }{(1-\overline{\kappa}) \, a_{0,\min}} ,
	\end{equation*}
	where the last inequality follows from the initial estimate \eqref{est:norm_u_a}. The estimate \eqref{eq:connection_norms} then gives
	\begin{align*}
		\sum_{\substack{\bsomega \le \bsnu \\ |\bsomega| = k}} \bsb^{2\bsomega} \|\partial^{\bsomega}u(\cdot,\bsy)\|_V^2
		&\le
		\frac{1}{(1-\overline{\kappa}) \, a_{0,\min}} \sum_{\substack{\bsomega \le \bsnu \\ |\bsomega| = k}} \bsb^{-2\bsomega} \|\partial^{\bsomega}u(\cdot,\bsy)\|_{a_{\bsy}}^2 \\
		&\le
		\left(\frac{2\widetilde{\kappa}(\bsnu)}{1-\overline{\kappa}}\right)^{2k} \frac{\|f\|_{V^*}^2 }{(1-\overline{\kappa})^2 \, a_{0,\min}^2},
	\end{align*}
	which yields the first claim. The second claim follows since the sum over the $\bsomega \le \bsnu$ with $|\bsomega|=|\bsnu|$ and $\bsnu \in \calF_1$ consists only of the term corresponding to $\bsomega=\bsnu$.
\end{proof}

\begin{corollary} \label{cor:deriv_bound} 
	Under the assumptions of Theorem \ref{thm:deriv_bound}, there exists a number $\kappa(k)$ for each $k \in \N$, given by
	\begin{equation*}
		\kappa(k) := \sup_{\substack{\bsnu \in \calF_1 \\ |\bsnu| = k}} \widetilde{\kappa}(\bsnu) ,
	\end{equation*}
	such that $\widetilde{\kappa}(\bsnu) \le \kappa(k) \le \kappa < 1$  for all $\bsnu \in \calF_1$  with $|\bsnu| = k$.
	Then for $\bsnu \in \calF_1$, every $f \in V^{\ast}$, and every $\bsy \in U$, the solution $u(\cdot,\bsy)$ satisfies
	\begin{equation} \label{eq:bound_deriv_pod}
		\|\partial^{\bsnu} u(\cdot,\bsy) \|_{V}
		\le
		\bsb^{\bsnu} \left(\frac{2\kappa(|\bsnu|)}{1-\overline{\kappa}}\right)^{|\bsnu|} \frac{\|f\|_{V^*}}{(1-\overline{\kappa}) \, a_{0,\min}} .
	\end{equation}
\end{corollary}
Note that since $0 < \overline{\kappa} \le \kappa < 1$, the results of Theorem \ref{thm:deriv_bound} and Corollary \ref{cor:deriv_bound}
remain also valid for $\overline{\kappa}$ replaced by $\kappa$. 

The obtained bounds on the mixed first-order derivatives turn out to be of product and order-dependent (so-called POD) form; that is, they are of the general form 
\begin{equation} \label{eq:form_pod_bounds}
	\|\partial^{\bsnu} u(\cdot,\bsy) \|_{V}
	\le
	C \, \bsb^{\bsnu} \, \Gamma(|\bsnu|) \, \|f\|_{V^{\ast}}
\end{equation}
with a map $\Gamma: \N_0 \to \R$, a sequence of reals $\bsb = (b_j)_{j \ge 1} \in \R^{\N}$ and some constant $C \in \R_{+}$.
This finding motivates us to consider this special type of bounds in the following error analysis.

\section{Quasi-Monte Carlo finite element error}

We analyze the error $\EE[G(u)] - Q_N(G(u_h^s))$ obtained by applying QMC
rules to the finite element approximation $u_h^s$ to approximate the expected value
\begin{equation*}
  \EE[G(u)] = \int_{U} G(u(\cdot,\bsy))\,\rd\bsy
  .
\end{equation*} 
To this end, we introduce the finite element approximation $u_h^s(\bsx,\bsy) := u_h(\bsx,(\bsy_{\{1:s\}};0))$ of a solution of \eqref{eq:PDE_weak} with truncated diffusion coefficient $a(\bsx,(\bsy_{\{1:s\}};0))$, where $u_h$ is a finite element approximation as defined in \eqref{eq:PDE_weak u_h} and $(\bsy_{\{1:s\}};0)=(y_1,\ldots,y_s,0,0,\ldots)$. The overall absolute QMC finite element error is then bounded as follows
	\begin{align}
	&| \EE[G(u)] - Q_N(G(u_h^s)) | \nonumber \\
	&\quad=
	| \EE[G(u)] - \EE[G(u^s)] + \EE[G(u^s)] - \EE[G(u_h^s)] + \EE[G(u_h^s)] - Q_N(G(u_h^s)) | \nonumber \\
	&\quad\le
	|\EE[G(u-u^s)]| + |\EE[G(u^s-u_h^s)]| + |\EE[G(u_h^s)] - Q_N(G(u_h^s))| \label{eq:error_split} .
	\end{align}
The first term on the right hand side of \eqref{eq:error_split} will be referred to as (dimension) truncation error, 
the second term is the finite element discretization error and the last term is the QMC quadrature error
for the integrand $u_h^s$. In the following sections we will analyze these different error terms separately.

\subsection{Finite Element Approximation}
\label{subsec:FEM}

Here, we consider the approximation of the solution $u(\cdot,\bsy)$ of \eqref{eq:PDE_weak} by a finite element
approximation $u_h(\cdot,\bsy)$ and assess the finite element discretization error. More specifically, denote by $\{V_h\}_{h>0}$ a family of subspaces $V_h \subset V$ of finite dimension $M_h$ such that $V_h \to V$ as $h \to 0$. We define the parametric finite element (FE) approximation as follows:
 for $f\in V^*$ and given $\bsy\in U$, find $u_h(\cdot,\bsy)\in V_h$ such that
\begin{equation} \label{eq:PDE_weak u_h}
	A(\bsy;u_h(\cdot,\bsy), v_h) = \langle f,v_h\rangle_{V^*\times V}
	= \int_D f(\bsx) v_h(\bsx)\,\rd\bsx \quad\mbox{for all}\quad v_h\in V_h .
\end{equation}
To establish convergence of the finite element approximations, we need some further conditions on $a(\bsx,\bsy)$.
To this end, we define the space $W^{1,\infty}(D) \subseteq L^\infty(D)$ endowed with the norm 
$\|v\|_{W^{1,\infty}(D)} = \max\{ \|v\|_{L^{\infty}(D)}, \|\nabla v\|_{L^{\infty}(D)} \}$ and require that
\begin{equation} \label{eq:cond_fem}
	a_0 \in W^{1,\infty}(D) \quad \text{and} \quad \sum_{j \ge 1} \|\psi_j\|_{W^{1,\infty}(D)} < \infty .
\end{equation}
Under these conditions and using that $f \in L^2(D)$, it was proven in \cite[Theorems 7.1 and 7.2]{KSS12} that for any $\bsy \in U$
the approximations $u_h(\cdot,\bsy)$ satisfy
\begin{equation*} \label{eq:uh_bound}
	\|u(\cdot,\bsy) - u_h(\cdot,\bsy)\|_V
	\le 
	C_1 \,h\, \|f\|_{L^2} .
\end{equation*}
In addition, if (the representer of) the bounded linear functional $G \in V^*$ lies in $L^2(D)$ we have for any $\bsy \in U$,
as $h \to 0$,
\begin{align} \label{eq:IG_uh_bound}
	|G(u(\cdot,\bsy)) - G(u_h(\cdot,\bsy))| \nonumber
	&\le 
	C_2 \,h^2\, \|f\|_{L^2}\, \|G\|_{L^2}, \\
	|\EE[G(u(\cdot,\bsy) - u_h(\cdot,\bsy))]| \,
	&\le 
	C_3 \,h^2\, \|f\|_{L^2}\, \|G\|_{L^2},
\end{align}
where the constants $C_1,C_2,C_3 > 0$ are independent of $h$ and $\bsy$. Since the above statements hold true for any $\bsy \in U$, they remain also valid for $u^s(\bsx,\bsy) := u(\bsx,(\bsy_{\{1:s\}};0))$ and $u_h^s(\bsx,\bsy) := u_h(\bsx,(\bsy_{\{1:s\}};0))$.

\subsection{Dimension Truncation}
\label{subsec:Truncation}

For every $s \in \N$ and $\bsy \in U$, we formally define the solution of the parametric weak problem \eqref{eq:PDE_weak} corresponding to the diffusion coefficient $a(\bsx,(\bsy_{\{1:s\}};0))$ with sum truncated to $s$ terms as
\begin{equation} \label{eq:truncated_sol}
	u^s(\cdot,\bsy) := u(\cdot, (\bsy_{\{1:s\}};0)) .
\end{equation}
In \cite[Proposition 5.1]{GHS18} it was shown that for the solution $u^s$ the following error estimates are satisfied.
\begin{theorem}
	Let $\overline{\kappa} \in (0,1)$ be such that \eqref{eq:kappa_bar} is satisfied and assume furthermore that there exists a 
	sequence of reals $\bsb = (b_j)_{j\ge1}$ with $0 < b_j \le 1$  for all $j$ and a constant $\kappa \in [\overline{\kappa},1)$ as defined
	in \eqref{def:kappa}. Then, for every $\bsy \in U$ and each $s \in \N$
	\begin{equation*}
		\|u(\cdot,\bsy) - u^s(\cdot,\bsy) \|_V
		\le
		\frac{a_{0,\max} \, \|f\|_{V^*}}{( a_{0,\min} (1-\overline{\kappa}))^2 } \sup_{j \ge s+1} b_j .
	\end{equation*}
	Moreover, if it holds for $\kappa$ that $\frac{\kappa \, a_{0,\max}}{(1-\overline{\kappa}) \, a_{0,\min}} \sup_{j \ge s+1} b_j < 1$,
	then for every $G \in V^*$ we have
	\begin{multline} \label{eq:truncation_IG_u}
		\left| \EE[G(u)] - \int_{\left[-\frac12,\frac12\right]^s} G(u^s(\cdot,(\bsy_{\{1:s\}};0))) \,\rd\bsy_{\{1:s\}} \right|
		\\\le
		\frac{\|G\|_{V^*} \, \|f\|_{V^*}}{(1-\overline{\kappa}) \, a_{0,\min} - a_{0,\max} \, \kappa \sup_{j \ge s+1} b_j}
		\left(\frac{a_{0,\max}}{(1-\overline{\kappa}) \, a_{0,\min}} \kappa \sup_{j \ge s+1} b_j \right)^2 .
	\end{multline}
\end{theorem}
In the following subsection, we will discuss how to approximate the finite-dimensional integral of solutions
of the form \eqref{eq:truncated_sol} by means of QMC methods. 

\subsection{Quasi-Monte Carlo Integration}
\label{subsec:QMC}

For a real-valued function $F:[-\tfrac12,\tfrac12]^s \to \R$ defined over the $s$-dimensional
unit cube centered at the origin, we consider the approximation of the integral $I_s(F)$ by $N$-point QMC
rules $Q_N(F)$, i.e., 
\begin{equation*}
	I_s(F)
	:= \int_{[-\frac12,\frac12]^s} F(\bsy)\, \rd\bsy
	\,\approx\,
	\frac1N \sum_{k=1}^N F(\bst_{k})
	=:
	Q_N(F) ,
\end{equation*}
with quadrature points $\bst_1,\ldots,\bst_N \in [-\tfrac12,\tfrac12]^s$.
As a quality criterion of such a rule, we define the worst-case error for QMC integration in some Banach space $\calH$ as  
\begin{equation*}
	e^{\text{wor}}(\bst_1,\ldots,\bst_N)
	:=
	\sup_{\substack{F\in\calH\\ \|F\|_{\calH}\le 1}} |I_s(F) - Q_N(F)| .
\end{equation*}
In this article, we consider randomly shifted rank-1 lattice rules as randomized QMC rules, with underlying points of the form
\begin{equation*}
	\widetilde{\bst}_k(\bsDelta) = \left\{(k \bsz)/N + \bsDelta \right\} - \left(1/2,\ldots,1/2\right), \quad k=1,\ldots,N ,
\end{equation*}
with generating vector $\bsz \in \Z^s$, uniform random shift $\bsDelta \in [0,1]^s$ and component-wise applied fractional part, 
denoted by $\{ \bsx \}$. For simplicity, we denote the worst-case error using a shifted lattice rule with generating 
vector $\bsz$ and shift $\bsDelta$ by $e_{N,s}(\bsz,\bsDelta)$. 

For randomly shifted QMC rules, the probabilistic error bound 
\begin{equation*}
	\sqrt{\bbE_{\bsDelta}\left[ |I_s(F) - Q_N(F)|^2 \right]}
	\le
	\widehat{e}_{N,s}(\bsz) \, \|F\|_{\calH},
\end{equation*}
holds for all $F \in \calH$, with shift-averaged worst-case error 
\begin{equation*}
	\widehat{e}_{N,s}(\bsz)
	:= 
	\left(\int_{[0,1]^s} e^2_{N,s}(\bsz,\bsDelta) \,\rd\bsDelta \right)^{1/2} .
\end{equation*}
As function space $\calH$ for our integrands $F$, we consider the weighted, unanchored Sobolev space
$\calW_{s,\bsgamma}$, which is a Hilbert space of functions defined over 
$[-\frac{1}{2},\frac{1}{2}]^s$ with square integrable mixed first derivatives and general non-negative weights 
$\bsgamma = (\gamma_{\setu})_{\setu \subseteq \{1:s\}}$. More precisely, the norm for $F \in \calW_{s,\bsgamma}$ 
is given by
\begin{equation} \label{eq:sob_norm}
	\|F\|_{\calW_{s,\bsgamma}}
	:=
	\left(
	\sum_{\setu\subseteq\{1:s\}}
	\gamma_\setu^{-1}
	\int_{[-\frac{1}{2},\frac{1}{2}]^{|\setu|}}
	\left(
	\int_{[-\frac{1}{2},\frac{1}{2}]^{s-|\setu|}}
	\frac{\partial^{|\setu|}F}{\partial \bsy_\setu}(\bsy_\setu;\bsy_{-\setu})
	\,\rd\bsy_{-\setu}
	\right)^2 \,
	\rd\bsy_\setu
	\right)^{1/2},
\end{equation}
where $\{1:s\} := \{1,\ldots,s\}$, $\frac{\partial^{|\setu|}F}{\partial \bsy_\setu}$ denotes the mixed first
derivative with respect to the variables $\bsy_\setu = (y_j)_{j\in\setu}$ and we set 
$\bsy_{-\setu} = (y_j)_{j\in\{1:s\}\setminus\setu}$. \\

For the efficient construction of good lattice rule generating vectors, we consider the so-called reduced
component-by-component (CBC) construction introduced in \cite{DKLP15}.
For $b \in \N$ and $m \in \N_0$, we define the group of units of integers modulo $b^m$ via
\begin{equation*}
  \ZZ_{b^m}^{\times} := \left\{z\in \Z_{b^m}: \gcd(z,b^m)=1\right\},
\end{equation*}
and note that $\ZZ_{b^0}^{\times}=\ZZ_{1}^{\times}=\{0\}$ since $\gcd(0,1)=1$. 
Henceforth, let $b$ be prime and recall that then, for $m \ge 1$, $|\ZZ_{b^m}^{\times}|=\varphi(b^m)=b^{m-1} \varphi(b)$ 
and $|\ZZ_{b}^\times| = \varphi(b) = (b-1)$, where $\varphi$ is Euler's totient function.
Let $\bsw:=(w_j)_{j\ge 1}$ be a non-decreasing sequence of integers in $\NN_0$, the elements of which we will refer to as reduction indices.
In the reduced CBC algorithm the components $\widetilde{z}_j$ of the generating vector $\widetilde{\bsz}$ of the lattice rule will be taken as multiples of $b^{w_j}$.

In \cite{DKLP15}, the reduced CBC construction was introduced to construct rank-1 lattice rules for $1$-periodic 
functions in a weighted Korobov space $\calH(K_{s,\alpha,\bsgamma})$ of smoothness $\alpha$
(see, e.g., \cite{SW01}). We denote the worst-case error in $\calH(K_{s,\alpha,\bsgamma})$ using a 
rank-1 lattice rule with generating vector $\bsz$ by $e_{N,s}(\bsz)$. Following \cite{DKLP15}, the reduced CBC construction 
is then given in Algorithm~\ref{alg:RedCBCAlg}.

{\centering
	\begin{minipage}{\linewidth}
		\begin{algorithm}[H]
			\small
			\caption{\small Reduced component-by-component construction}	
			\label{alg:RedCBCAlg}
			\textbf{Input:} Prime power $N=b^m$ with $m \in \N_0$ and integer reduction indices $0 \le w_1 \le \cdots \le w_s$. \\[1.75mm]
			For $j$ from $1$ to $s$ and as long as $w_j<m$ do:
			\begin{itemize}
				\item[]
				\begin{itemize}
					\item[$\bullet$] Select $z_j \in \ZZ_{b^{m-w_j}}^{\times}$ such that
					\vspace{-7pt}
					\begin{equation*}
						z_j := \argmin_{z \in \ZZ_{b^{m-w_j}}^{\times}} e^2_{N,j}(b^{w_1} z_1,\ldots,b^{w_{j-1}} z_{j-1}, b^{w_j} z) .
					\end{equation*}
				\end{itemize}
			\end{itemize}
			Set all remaining $z_j := 0$ (for $j$ with $w_j \ge m$). \\[1.75mm]
			\textbf{Return:} Generating vector $\widetilde{\bsz}:=(b^{w_1} z_1,\ldots, b^{w_s} z_s)$ for $N=b^m$.
		\end{algorithm}
	\end{minipage}
}
\par
\vspace{5pt}
The following theorem, proven in \cite{DKLP15}, states that the algorithm yields generating 
vectors with a small integration error for general weights $\gamma_{\setu}$ in the Korobov space.

\begin{theorem} \label{thm:redcbc_korobov}
	For a prime power $N=b^m$ let $\widetilde{\bsz}=(b^{w_1} z_1,\ldots, b^{w_s} z_s)$  be constructed according to Algorithm~\ref{alg:RedCBCAlg} with integer reduction indices $0 \le w_1 \le \cdots \le w_s$. Then for every $d\in \{1:s\}$ and every $\lambda \in(1/\alpha,1]$ it holds for the worst-case error in the Korobov space $\calH(K_{s,\alpha,\bsgamma})$ with $\alpha > 1$ that
	\begin{equation*}
		e_{N,d}^2(b^{w_1} z_1,\ldots, b^{w_d} z_d)
		\le
		\left(\sum_{\emptyset\neq\setu\subseteq \{1:d\}} \gamma_\setu^\lambda 
		\, (2\zeta (\alpha\lambda))^{|\setu|} \, b^{\min\{m,\max_{j\in\setu}w_j\}}\right)^{\frac{1}{\lambda}}
		\left(\frac{2}{N}\right)^{\frac1\lambda}.
	\end{equation*}
\end{theorem}

This theorem can be extended to the weighted unanchored Sobolev space
$\calW_{s,\bsgamma}$ using randomly shifted lattice rules as follows.

\begin{theorem} \label{thm:QMC_sh_lat_red_CBC}
	For a prime power $N = b^m$, $m \in \NN_0$, and for $F \in \calW_{s,\bsgamma}$ belonging to the weighted unanchored Sobolev space defined over $[-\frac12,\frac12]^s$ with weights $\bsgamma = (\gamma_{\setu})_{\setu \subseteq \{1:s\}}$, a randomly shifted lattice rule can be constructed by the reduced CBC algorithm, see Algorithm~\ref{alg:RedCBCAlg}, such that for all $\lambda\in (1/2,1]$,
	\begin{multline*}
		\sqrt{\bbE_{\bsDelta}\left[ |I_s(F) - Q_N(F)|^2 \right]} \\
		\le
		\left(
		\sum_{\emptyset\ne\setu\subseteq\{1:s\}} \gamma_\setu^\lambda\,
		\varrho^{|\setu|}(\lambda) \, b^{\min\{m,\max_{j \in \setu} w_j\}}
		\right)^{1/(2\lambda)}
		\left( \frac2N \right)^{1/(2\lambda)} \, \|F\|_{\calW_{s,\bsgamma}},
	\end{multline*}
	with integer reduction indices $0 \le w_1 \le \cdots \le w_s$ and 
	$\varrho(\lambda) = 2\zeta(2\lambda) (2\pi^2)^{-\lambda}$.
\end{theorem}
\begin{proof}
	Using Theorem~\ref{thm:redcbc_korobov} and the connection that the shift-averaged kernel of the Sobolev space equals the kernel of the Korobov space 
    $\calH(K_{s,\alpha,\widetilde{\bsgamma}})$ with $\alpha = 2$ and weights $\widetilde{\gamma}_\setu = \gamma_\setu / (2\pi^2)^{|\setu|}$, see, 
  	e.g., \cite{DKS13,N2014}, the result follows from
	\begin{equation*}
		\sqrt{\bbE_{\bsDelta}\left[ |I(F) - Q_N(F)|^2 \right]}
		\le
		\sqrt{\bbE_{\bsDelta}\left[ e^2_{N,s}(\bsz,\bsDelta) \, \|F\|_{\calW_{s,\bsgamma}}^2 \right]}
		=
		\widehat{e}_{N,s}(\bsz) \, \|F\|_{\calW_{s,\bsgamma}} . \qedhere
	\end{equation*} 
\end{proof}

It follows that we can construct the lattice rule in the weighted Korobov space using the connection mentioned in the proof of the previous theorem.

\subsection{Implementation of the reduced CBC algorithm}
\label{sec:fast-reduced-CBC-POD}

Similar to other variants of the CBC construction, we present a fast version of the reduced CBC method for POD weights in Algorithm~\ref{alg:RedCBCAlgPOD} for which Theorems~\ref{thm:redcbc_korobov} and~\ref{thm:QMC_sh_lat_red_CBC} still hold. The full derivation of Algorithm~\ref{alg:RedCBCAlgPOD} is given in Section~\ref{sec:appendix},
here we only introduce the necessary notation.
The squared worst-case error for POD weights $\bsgamma = (\gamma_{\setu})_{\setu \subseteq \{1:s\}}$ with 
$\gamma_{\setu} = \Gamma(|\setu|) \prod_{j \in \setu} \gamma_j$ and $\gamma_\emptyset = 1$ in the weighted Korobov space 
$\calH(K_{s,\alpha,\bsgamma})$ with $\alpha > 1$ can be written as
\begin{equation*}
	e^2_{N,s}(\bsz)
	=
	\frac1N \sum_{k=0}^{N-1} \sum_{\ell=1}^{s} \sum_{\substack{\setu \subseteq \{1:s\} \\ \abs{\setu}=\ell}} 
	\Gamma(\ell) \prod_{j \in \setu} \gamma_j \, \omega\!\left(\left\{\frac{k z_j}{N}\right\}\right) ,
\end{equation*}
where $\omega(x) = \sum_{0 \ne h \in \Z} \rme^{2\pi \rmi \, h x} / \abs{h}^{\alpha}$, see, e.g., \cite{DKS13,N2014}, and for $n \in \N$ we define $\Omega_n$ as
\begin{equation*}
	\Omega_{n}
	:=
	\left[ \omega\!\left( \frac{k z \bmod n}{n} \right) \right]_{\substack{z \in \ZZ_{n}^{\times} \\ k \in \ZZ_{n}}}
	\in \R^{\varphi(n) \times n}
	.
\end{equation*}
We assume that the values of the function $\omega$ can be computed at unit cost.
For integers $0 \le w' \le w'' \le m$ and given base~$b$ we define the ``fold and sum'' operator, which divides a length $b^{m-w'}$ 
vector into blocks of equal length $b^{m-w''}$ and sums them up, i.e.,
\begin{align}\label{eq:fold-and-sum}
  P_{w'',w'}^m
  :
  \R^{b^{m-w'}} \to \R^{b^{m-w''}}
  :
  P_{w'',w'}^m \, \bsv
  =
  \bigl[ \; \underbrace{I_{b^{m-w''}} | \cdots | I_{b^{m-w''}}}_{b^{w''-w'} \text{ times}} \; \bigr] \, \bsv
  ,
\end{align}
where $\cramped{I_{b^{m-w''}}}$ is the identity matrix of size $\cramped{b^{m-w''} \times b^{m-w''}}$.
The computational cost of applying $P_{w'',w'}^m$ is the length of the input vector $\cramped{\calO(b^{m-w'})}$.
It should be clear that $\cramped{P_{w''',w''}^m \, P_{w'',w'}^m \, \bsv = P_{w''',w'}^m \, \bsv}$ for $0 \le w' \le w'' \le w''' \le m$.
In step~\ref{step-update} of Algorithm \ref{alg:RedCBCAlgPOD} the notation $.*$ denotes the element-wise product of 
two vectors and $\Omega_{b^{m-w_j}}(z_j,:)$ means to take the row corresponding to $z=z_j$ from the matrix. Furthermore, 
Algorithm \ref{alg:RedCBCAlgPOD} includes an optional step in which the reduction indices are adjusted in case $w_1 > 0$, the 
auxiliary variable $w_0=0$ is introduced to satisfy the recurrence relation.

The standard fast CBC algorithm for POD weights has a complexity of $\calO(s \, N \log N + s^2 N)$, see, e.g., \cite{DKS13,N2014}.
The cost of our new algorithm can be substantially lower as is stated in the following theorem. We stress that the presented algorithm 
is the first realization of the reduced CBC construction for POD weights. Our new algorithm improves upon the one stated in \cite{DKLP15} 
which only considers product weights, but the same technique can be used there since POD weights are more general and include product weights.

\begin{theorem} \label{thm:complexity_redcbc}
	Given a sequence of integer reduction indices $0 \le w_1 \le w_2 \le \cdots$, the reduced CBC algorithm for a prime power $N = b^m$ points in $s$ dimensions 
	as specified in Algorithm~\ref{alg:RedCBCAlgPOD} can construct a lattice rule with near optimal worst-case error as in Theorem~\ref{thm:QMC_sh_lat_red_CBC} with an arithmetic cost of
	\begin{equation*}
	  \calO\!\left(\sum_{j=1}^{\min\{s,s^{\ast}\}} (m-w_j+j) \, b^{m-w_j} \right),
	\end{equation*}
	where $s^*$ is defined to be the largest integer such that $w_{s^*} < m$. The memory cost is $\calO(\sum_{j=1}^{\min\{s,s^{\ast}\}} b^{m-w_j})$.
	In case of product weights $\calO(\sum_{j=1}^{\min\{s,s^{\ast}\}} (m-w_j) \, b^{m-w_j})$ operations are required for the construction with 
	memory $\calO(b^{m-w_1})$.
\end{theorem}
\begin{proof}
	We refer to Algorithm~\ref{alg:RedCBCAlgPOD}.
	Step~\ref{step-qbar} can be calculated in $\calO(j\,b^{m-w_{j-1}})$ operations (and we may assume $w_0 = w_1$ since the case $w_1 > 0$ 
	can be reduced to the case $w_1 = 0$). The matrix-vector multiplication in step~\ref{step-mv} can be done by exploiting the block-circulant structure to obtain a fast matrix-vector product by FFTs at a cost of $\calO((m-w_j) \, b^{m-w_j})$, see, e.g., \cite{CKN06,CN06b}.
	We ignore the possible saving by pre-computation of FFTs on the first columns of the blocks in the matrices $\Omega_{b^{m-w_j}}$ as this has cost $\calO((m-w_1) \, b^{m-w_1})$ and therefore is already included in the cost of step~\ref{step-mv}. Finally, the vectors $\bsq_{j,\ell}$ for $\ell=1,\ldots,j$ in step~\ref{step-update} can be calculated in $\calO(j \,b^{m-w_{j-1}})$. To obtain the total complexity we remark that the applications of the ``fold and sum'' operator, marked by the square brackets could be performed in iteration $j-1$ such that the cost of steps~\ref{step-qbar} and~\ref{step-update} in iteration $j$ are only $\calO(j \, b^{m-w_j})$ instead of $\calO(j \, b^{m-w_{j-1}})$. The cost of the 
	additional fold and sum to prepare for iteration $j$ in iteration $j-1$, which can be performed after step~\ref{step-update}, is then equal to the cost of step~\ref{step-update} in 
	that iteration. Since we can assume $w_0 = w_1$ we obtain the claimed construction cost. Note that the algorithm is written in such a way that the vectors $\bsq_{j,\ell-1}$ can be reused for storing the vectors $\bsq_{j,\ell}$ (which might be smaller). Similarly for the vectors $\overline{\bsq}_j$.
	Therefore the memory cost is $\calO(\sum_{j=1}^{\min\{s,s^{\ast}\}} b^{m-w_j})$. The result for product weights can be obtained similarly, see, e.g., \cite{N2014}.
\end{proof}

{\centering
\begin{minipage}{\linewidth}
	\begin{algorithm}[H]
		\small
		\caption{\small Fast reduced CBC construction for POD weights}	
		\label{alg:RedCBCAlgPOD}
  		\textbf{Input:} Prime power $N=b^m$ with $m \in \N_0$, integer reduction indices $0 \le w_1 \le \cdots \le w_s$, \\
   		and weights $\Gamma(\ell)$, $\ell \in \N_0$ with $\Gamma(0) = 1$, and $\gamma_j$, $j \in \N$ such that $\gamma_\setu = \Gamma(|\setu|) \prod_{j\in\setu} \gamma_j$.
		\\[2mm]
  		Optional: Adjust $m := \max\{0, m - w_1\}$ and for $j$ from $s$ down to $1$ adjust $w_j := w_j - w_1$.
  		\\[1mm]
  		Set $\bsq_{0,0}:=\mathbf{1}_{b^m}$ and $\bsq_{0,1}:=\mathbf{0}_{b^m}$, set $w_0:=0$. 
  		\\[1mm]
 		For $j$ from $1$ to $s$ and as long as $w_j < m$ do:
		\begin{enumerate}
			\item\label{step-qbar} Set $\overline{\bsq}_j := \sum_{\ell=1}^{j} \frac{\Gamma(\ell)}{\Gamma(\ell-1)} \left[ P_{w_j,w_{j-1}}^{m} \bsq_{j-1,\ell-1} \right] \in \R^{b^{m-w_j}}$ 
			(with $\bsq_{j-1,\ell-1} \in \R^{b^{m-w_{j-1}}}$).
			\item\label{step-mv} Calculate $\bsT_j := \Omega_{b^{m-w_j}} \, \overline{\bsq}_j \in \R^{\varphi(b^{m-w_j})}$ by exploiting the block-circulant structure 
			of the matrix $\Omega_{b^{m-w_j}}$ using FFTs.
			\item\label{step-argmin} Set $z_j := \argmin_{z \in \ZZ_{b^{m-w_j}}^{\times}} \bsT_j(z)$, with $\bsT_j(z)$ the component corresponding to $z$.
			\item\label{step-update} Set $\bsq_{j,0} := \mathbf{1}_{b^{m-w_j}}$ and $\bsq_{j,j+1} := \mathbf{0}_{b^{m-w_j}}$ and for $\ell$ from $j$ down to $1$ set
			\begin{equation*}
				\bsq_{j,\ell} := \left[ P_{w_j,w_{j-1}}^{m} \bsq_{j-1,\ell} \right] + \frac{\Gamma(\ell)}{\Gamma(\ell-1)} \gamma_j \, \Omega_{b^{m-w_j}}(z_j,:) \, {.*} \, \left[ P_{w_j,w_{j-1}}^{m} \bsq_{j-1,\ell-1} \right] \in \R^{b^{m-w_j}}.
			\end{equation*}
			\item\label{step-wce} Optional: Calculate squared worst-case error by $e^2_j := \frac1{b^m} \sum_{k \in \Z_{b^{m-w_j}}} \sum_{\ell=1}^j \bsq_{j,\ell}(k)$.
		\end{enumerate}
	Set all remaining $z_j := 0$ (for $j$ with $w_j \ge m$). \\[1.75mm]
	\textbf{Return:} Generating vector $\widetilde{\bsz}:=(b^{w_1} z_1,\ldots, b^{w_s} z_s)$ for $N=b^m$. \\
	(Note: the $w_j$'s and $m$ might have been adjusted to make $w_1=0$.)
		\end{algorithm}
\end{minipage}
}

\section{QMC finite element error analysis}

We now combine the results of the previous subsections to analyze the overall QMC finite element error. We consider 
the root mean square error (RMSE) given by
\begin{equation*}
	e_{N,s,h}^{\text{RMSE}}(G(u))
	:=
	\sqrt{\EE_{\bsDelta}\left[|\EE[G(u)] - Q_N(G(u_h^s))|^2\right]}
	.
\end{equation*}
The error $\EE[G(u)] - Q_N(G(u_h^s))$ can be written as
\begin{align*}
	\EE[G(u)] - Q_N(G(u_h^s))
	&=
	\EE[G(u)] - I_s(G(u_h^s)) + I_s(G(u_h^s)) - Q_N(G(u_h^s))
\end{align*}
such that due to the fact that $\EE_{\bsDelta}(Q_N(f)) = I_s(f)$ for any integrand $f$ we obtain
\begin{align*}
	\EE_{\bsDelta}\left[(\EE[G(u)] - Q_N(G(u_h^s)))^2\right]
	&=
	(\EE[G(u)] - I_s(G(u_h^s)))^2 + \EE_{\bsDelta}\left[(I_s - Q_N)^2(G(u_h^s))\right] \\
	&\phantom{= } + 
	2 (\EE[G(u)] - I_s(G(u_h^s))) \, \EE_{\bsDelta}\left[(I_s - Q_N)(G(u_h^s))\right] \\
	&=
	(\EE[G(u)] - I_s(G(u_h^s)))^2 + \EE_{\bsDelta}\left[(I_s - Q_N)^2(G(u_h^s))\right] .
\end{align*}
Then, noting that $\EE[G(u)] - I_s(G(u_h^s)) = \EE[G(u)] - I_s(G(u^s)) + I_s(G(u^s)) - I_s(G(u_h^s))$,
\begin{align*}
	(\EE[G(u)] - I_s(G(u_h^s)))^2
	&=
	(\EE[G(u)] - I_s(G(u^s)))^2 + (I_s(G(u^s)) - I_s(G(u_h^s)))^2 \\
	&\phantom{= } +
	2 (\EE[G(u)] - I_s(G(u^s))) (I_s(G(u^s)) - I_s(G(u_h^s)))
\end{align*}
and since for general $x,y \in \RR$ it holds that $2 x y \le x^2 + y^2$, we obtain
furthermore
\begin{align*}
	(\EE[G(u)] - I_s(G(u_h^s)))^2
	&\le
	2 (\EE[G(u)] - I_s(G(u^s)))^2 + 2 (I_s(G(u^s)) - I_s(G(u_h^s)))^2 .
\end{align*}
From the previous subsections we can then use \eqref{eq:truncation_IG_u} for the truncation part, 
\eqref{eq:IG_uh_bound}, which holds for general $\bsy \in U$ and thus also for $y_{\{1:s\}}$, for the finite element error,
and Theorem \ref{thm:QMC_sh_lat_red_CBC} for the QMC integration error to obtain the following error bound 
for the mean square error $\EE_{\bsDelta}[ |\EE[G(u)] - Q_N(G(u_h^s))|^2 ] =: e^{\text{MSE}}_{N,s,h}(G(u))$,
\begin{align} \label{eq:combined_error}
	e^{\text{MSE}}_{N,s,h}(G(u))
	&\le
	K_1 \|f\|^2_{V^*}\|G\|^2_{V^*} \left(\frac{1}{(1-\overline{\kappa}) \, a_{0,\min} - a_{0,\max} \, \kappa \sup_{j \ge s+1} b_j} \right)^2 \nonumber \\ 
	&\phantom{\le}\times \left(\frac{a_{0,\max}}{(1-\overline{\kappa}) \, a_{0,\min}} \kappa \sup_{j \ge s+1} b_j \right)^4
	+ K_2 \|f\|^2_{L^2}\|G\|^2_{L^2} \, h^4 \\
	&\phantom{\le}+ \left(
	\sum_{\emptyset\ne\setu\subseteq\{1:s\}} \gamma_\setu^\lambda\,
	\varrho^{|\setu|}(\lambda) \, b^{\min\{m,\max_{j \in \setu} w_j\}}
	\right)^{1/\lambda}
	\left( \frac2N \right)^{1/\lambda} \, \|G (u_h^s)\|_{\calW_{s,\bsgamma}}^2 \nonumber
\end{align}
for some constants $K_1, K_2 \in \R_{+}$ and provided that 
$\frac{a_{0,\max}}{(1-\overline{\kappa}) \, a_{0,\min}} \kappa \, \sup_{j \ge s+1} b_j < 1$.

\subsection{Derivative bounds of POD form}

In the following we assume that we have general bounds on the mixed partial derivatives 
$\partial^{\bsnu} u(\cdot,\bsy)$ which are of POD form; that is,
\begin{equation} \label{eq:form_pod_bounds}
	\|\partial^{\bsnu} u(\cdot,\bsy) \|_{V}
	\le
	C \, \widetilde{\bsb}^{\bsnu} \, \Gamma(|\bsnu|) \, \|f\|_{V^{\ast}}
\end{equation}
with a map $\Gamma: \N_0 \to \R$, a sequence of reals $\widetilde{\bsb} = (\widetilde{b}_j)_{j \ge 1} \in \R^{\N}$ and some constant $C \in \R_{+}$.
Such bounds can be found in the literature and we provided a new derivation in Theorem~\ref{thm:deriv_bound} also leading to POD weights.

For bounding the norm $\|G (u_h^s)\|_{\calW_{s,\bsgamma}}$, we can then use \eqref{eq:form_pod_bounds} and the definition in \eqref{eq:sob_norm} to proceed as outlined in \cite{KN16}, to obtain the estimate 
\begin{equation} \label{eq:bound_functional}
	\|G(u_h^s)\|_{\calW_{s,\bsgamma}}
	\le 
	C \, \|f\|_{V^*} \|G\|_{V^*}
	\Bigg(
	\sum_{\setu\subseteq\{1:s\}} \frac{\Gamma(|\setu|)^2 \prod_{j\in\setu} \widetilde{b}_j^2}{\gamma_\setu}
	\Bigg)^{1/2}\,.
\end{equation}
Denoting $\bsw:=(w_j)_{j\ge 1}$ and using \eqref{eq:bound_functional}, the contribution of the quadrature error 
to the mean square error $e_{N,h,s}^{\text{MSE}}(G(u))$ can be upper bounded by
\begin{equation} \label{eq:intermquaderror}
	\begin{aligned} 
		&\left(
		\sum_{\emptyset\ne\setu\subseteq\{1:s\}} \gamma_\setu^\lambda\,
		\varrho^{|\setu|}(\lambda) \, b^{\min\{m,\max_{j \in \setu} w_j\}}
		\right)^{1/\lambda}
		\left( \frac2N \right)^{1/\lambda} \, \|G (u_h^s)\|^2_{\calW_{s,\bsgamma}} \\
		&\qquad\le
		C \, \|f\|_{V^*} \|G\|_{V^*} \, C_{\bsgamma,\bsw,\lambda} \left( \frac2N \right)^{1/\lambda} ,
	\end{aligned}
\end{equation}
where we define
\begin{equation*}
	C_{\bsgamma,\bsw,\lambda}
	:=
	\left(
	\sum_{\emptyset\ne\setu\subseteq\{1:s\}} \gamma_\setu^\lambda\,
	\varrho^{|\setu|}(\lambda) \, b^{\min\{m,\max_{j \in \setu} w_j\}}
	\right)^{1/\lambda} 
	\left(
	\sum_{\setu\subseteq\{1:s\}} \frac{\Gamma(|\setu|)^2 \prod_{j\in\setu} \widetilde{b}_j^2}{\gamma_\setu}
	\right) .
\end{equation*}
The term $C_{\bsgamma,\bsw,\lambda}$ can be bounded as
\begin{equation*}
	C_{\bsgamma,\bsw,\lambda}\le 
	\left(
	\sum_{\setu\subseteq\{1:s\}} \gamma_\setu^\lambda\,
	\varrho^{|\setu|}(\lambda) \, b^{\sum_{j \in \setu} w_j - \sum_{\ell=1}^{|\setu|-1} w_{\ell}}
	\right)^{1/\lambda} 
	\left(
	\sum_{\setu\subseteq\{1:s\}} \frac{\Gamma(|\setu|)^2 \prod_{j\in\setu} \widetilde{b}_j^2}{\gamma_\setu}
	\right) .
\end{equation*}
Due to \cite[Lemma~6.2]{KSS12} the latter term is minimized by choosing the weights 
$\gamma_\setu$ as
\begin{equation} \label{eq:form_weights}
	\gamma_\setu
	:=
	\left(\frac{\Gamma(|\setu|)^2\, \prod_{j\in\setu} \widetilde{b}_j^2 \, \prod_{\ell=1}^{|\setu|-1} b^{w_{\ell}}}{\prod_{j\in\setu} \rho (\lambda)\, b^{w_j}}\right)^{1/(1+\lambda)}.
\end{equation}
Then we set
\begin{equation*}
	A_{\lambda} 
	:=
	\sum_{\setu\subseteq\{1:s\}} \gamma_\setu^\lambda\,
	\varrho^{|\setu|}(\lambda) \, b^{\sum_{j \in \setu} w_j - \sum_{\ell=1}^{|\setu|-1} w_{\ell}}
	=
	\sum_{\setu\subseteq\{1:s\}} \left[ \left(\frac{\Gamma(|\setu|)^{2\lambda}}{\prod_{\ell=1}^{|\setu|-1} b^{w_{\ell}}}\right)
	\left(\prod_{j \in \setu} \rho(\lambda) \, \widetilde{b}_j^{2\lambda} \, b^{w_j} \right)
	\right]^\frac{1}{1+\lambda}
\end{equation*}
and easily see that also
\begin{equation*}
	\sum_{\setu\subseteq\{1:s\}} \gamma_\setu^{-1} \left(\Gamma(|\setu|)^2 \prod_{j\in\setu} \widetilde{b}_j^2\right)
	= A_\lambda,
\end{equation*}
which implies that $C_{\bsgamma,\bsw,\lambda} \le A_\lambda^{1 + 1/\lambda}$.
We demonstrate how the term $A_{\lambda}$ can be estimated for the derivative bounds 
derived in Section \ref{subsec:PDE}. 	

In view of Theorem \ref{thm:deriv_bound}, assume in the following that 
\begin{equation}\label{eq:assumpThm1}
	\Gamma (\abs{\setu})=\kappa^{\abs{\setu}},
	\quad
	\widetilde{b}_j=\frac{2\,b_j}{1-\kappa},
	\quad
	\sum_{j=1}^\infty \left(b_j b^{w_j}\right)^{p}<\infty \quad \text{for} \quad p \in (0,1) .
\end{equation}

Note that we could also choose $\Gamma (\abs{\setu})=\kappa (\abs{\setu})^{\abs{\setu}}$ above, 
in which case the subsequent estimate of $A_\lambda$ can be done analogously, but
to make the argument less technical, we consider the slightly coarser variant 
$\Gamma (\abs{\setu})=\kappa^{\abs{\setu}}$ here.
In this case, 
\[
	A_\lambda
	=
	\sum_{\setu\subseteq\{1:s\}} \left[\kappa^{\abs{\setu}}\right]^{\frac{2\lambda}{1+\lambda}}
	\left(\prod_{\ell=1}^{|\setu|-1} b^{\frac{-w_{\ell}}{2\lambda}}\right)^\frac{2\lambda}{1+\lambda}
	\prod_{j\in\setu} \left(\left(\frac{2\,b_j}{1-\kappa}\right)^{2\lambda}\, b^{w_j}\, \rho (\lambda)\right)^{\frac{1}{1+\lambda}}.
\]
Note that, as $\lambda \le 1$, it holds that $b^{\frac{-w_{\ell}}{2\lambda}} \le b^{\frac{-w_{\ell}}{2}}$ and hence
\[
	A_\lambda
	\le
	\sum_{\setu\subseteq\{1:s\}} \left(\kappa^{\abs{\setu}}
	\prod_{\ell=1}^{|\setu|-1} b^{\frac{-w_{\ell}}{2}}\right)^\frac{2\lambda}{1+\lambda} 
	\prod_{j\in\setu} \left(\left(\frac{2\,b_j}{1-\kappa}\right)^{2\lambda}\, b^{w_j}\, \rho (\lambda)\right)^{\frac{1}{1+\lambda}}.
\]
We now proceed similarly to the proof of Theorem 6.4 in \cite{KSS12}. Let $(\alpha_j)_{j\ge 1}$ be a sequence of positive reals, 
to be specified below, which satisfies $\Sigma:=\sum_{j=1}^\infty \alpha_j<\infty$. 
Dividing and multiplying by $\prod_{j\in\fraku} \alpha_j^{(2\lambda)/(1+\lambda)}$, 
and applying H\"{o}lder's inequality with conjugate components $p=(1+\lambda)/(2\lambda)$ and $p^*=(1+\lambda)/(1-\lambda)$,
\begin{eqnarray*}
	A_\lambda
	&\le&
	\sum_{\setu\subseteq\{1:s\}} \left(\kappa^{\abs{\setu}}
	\prod_{\ell=1}^{|\setu|-1} b^{\frac{-w_{\ell}}{2}}\right)^\frac{2\lambda}{1+\lambda}
	\left(\prod_{j\in\fraku} \alpha_j^{\frac{2\lambda}{1+\lambda}}\right)
	\prod_{j\in\setu} \left(\left(\frac{2\,b_j}{1-\kappa}\right)^{2\lambda}\, b^{w_j}\, \rho (\lambda)/\alpha_j^{2\lambda}\right)^{\frac{1}{1+\lambda}}\\
	&\le&
	\left(\sum_{\setu\subseteq\{1:s\}} \kappa^{\abs{\setu}}
	\left(\prod_{\ell=1}^{|\setu|-1} b^{\frac{-w_{\ell}}{2}}\right) \prod_{j\in\fraku} \alpha_j\right)^\frac{2\lambda}{1+\lambda} \\
	&&\times
	\left(\sum_{\setu\subseteq\{1:s\}} \prod_{j\in\setu} 
	\left(\left(\frac{2\,b_j}{1-\kappa}\right)^{2\lambda}\, b^{w_j}\, \rho (\lambda)/\alpha_j^{2\lambda}\right)^{\frac{1}{1-\lambda}}
	\right)^{\frac{1-\lambda}{1+\lambda}} = B^{\frac{2\lambda}{1+\lambda}} \cdot \widetilde{B}^{\frac{1-\lambda}{1+\lambda}} ,
\end{eqnarray*}
where we define
\[
	B
	:=
	\sum_{\setu\subseteq\{1:s\}} \kappa^{\abs{\setu}}
	\left(\prod_{\ell=1}^{|\setu|-1} b^{\frac{-w_{\ell}}{2}}\right) \prod_{j\in\fraku} \alpha_j,
	\quad
	\widetilde{B}
	:=
	\sum_{\setu\subseteq\{1:s\}} \prod_{j\in\setu} 
	\left(\left(\frac{2\,b_j}{1-\kappa}\right)^{2\lambda}\, \frac{b^{w_j}\, \rho (\lambda)}{\alpha_j^{2\lambda}} \right)^{\frac{1}{1-\lambda}} .
\]
For the first factor we estimate
\begin{eqnarray*}
	B
	&\le&
	\sum_{\setu:\, \abs{\setu}<\infty} \kappa^{\abs{\setu}} \prod_{\ell=1}^{|\setu|-1} b^{\frac{-w_{\ell}}{2}} \prod_{j\in\setu} \alpha_j
	=
	\sum_{k=1}^\infty \left( \kappa^{k} \prod_{\ell=1}^{k-1} b^{\frac{-w_{\ell}}{2}}\right)
	\sum_{\substack{\setu:\, \abs{\setu}<\infty\\ \abs{\setu}=k}} \prod_{j\in\fraku} \alpha_j \\
	&\le&
	\sum_{k=1}^\infty \left( \kappa^{k} \prod_{\ell=1}^{k-1} b^{\frac{-w_{\ell}}{2}}\right)\frac{1}{k!}
	\sum_{\bsu \in\NN^k} \prod_{i=1}^k \alpha_{u_i}
	=
	\sum_{k=1}^\infty \left( \kappa^{k}	\prod_{\ell=1}^{k-1} b^{\frac{-w_{\ell}}{2}}\right)\frac{1}{k!} \Sigma^k.
\end{eqnarray*}
By the ratio test, the latter expression is finite if we choose $(\alpha_j)_{j\ge 1}$ such that 
$L:=\sup_{k\in\NN} \kappa\,b^{\frac{-w_k}{2}}(k+1)^{-1}=\kappa b^{\frac{-w_1}{2}}/2<1/\Sigma$.
Hence we assume that $(\alpha_j)_{j\ge 1}$ is chosen such that indeed $L<1/\Sigma$. 
Note that $L$ is small if $\kappa$ is small, which means that $\Sigma$ can be allowed to 
be large in this case. Consider now the term
\begin{eqnarray*}
	\widetilde{B} &\le& \sum_{\setu:\, \abs{\setu}<\infty} 
	\prod_{j\in\setu} 
	\left(\left(\frac{2\,b_j}{1-\kappa}\right)^{2\lambda}\, b^{w_j}\, \rho (\lambda)/\alpha_j^{2\lambda}\right)^{\frac{1}{1-\lambda}}\\
	&\le & \exp\left(\sum_{j=1}^\infty
	\left(\left(\frac{2\,b_j}{1-\kappa}\right)^{2\lambda}\, b^{w_j}\, \rho (\lambda)/\alpha_j^{2\lambda}\right)^{\frac{1}{1-\lambda}}\right)\\
	&\le& \exp \left(\sum_{j=1}^\infty \left(\frac{1}{1-\kappa}\right)^{\frac{2\lambda}{1-\lambda}}(\rho(\lambda))^{\frac{1}{1-\lambda}}4^\lambda 
	\left(b_j b^{w_j}\frac{1}{\alpha_j}\right)^{\frac{2\lambda}{1-\lambda}}\right)\\
	&=& \exp\left(\left(1-\kappa\right)^{\frac{-2\lambda}{1-\lambda}}(\rho(\lambda))^{\frac{1}{1-\lambda}}4^\lambda
	\sum_{j=1}^\infty \left(b_j b^{w_j}\alpha_j^{-1}\right)^{\frac{2\lambda}{1-\lambda}}\right).
\end{eqnarray*}
We require
\begin{equation}\label{eq:condAlambda}
	L<1/\Sigma=1/\sum_{j=1}^\infty \alpha_j\quad\mbox{and}
	\quad \sum_{j=1}^\infty \left(b_j b^{w_j}\alpha_j^{-1}\right)^{\frac{2\lambda}{1-\lambda}}<\infty.
\end{equation}
To this end, we choose
$\alpha_j:=\frac{\left(b_j b^{w_j}\right)^{p}}{\theta}$, where 
$\frac{\theta}{L}>\sum_{j=1}^\infty \left(b_j b^{w_j}\right)^{p}$.
Then, 
\begin{eqnarray}\label{eq:Alambdabound}
	A_\lambda &\le & \left(\sum_{k=1}^\infty \left( \kappa^{k}
	\prod_{\ell=1}^{k-1} b^{\frac{-w_{\ell}}{2}}\right)\frac{1}{k!}
	\Sigma^k\right)^\frac{2\lambda}{1+\lambda}\nonumber\\
	&&\times \exp\left(\frac{1-\lambda}{1+\lambda}
	\left(\frac{1}{1-\kappa}\right)^{\frac{2\lambda}{1-\lambda}}(\rho(\lambda))^{\frac{1}{1-\lambda}}4^\lambda
	\sum_{j=1}^\infty \left(b_j b^{w_j}\frac{1}{\alpha_j}\right)^{\frac{2\lambda}{1-\lambda}}\right)
\end{eqnarray}
as long as we choose $\lambda$ such that
\begin{equation}\label{eq:chooselambda}
\sum_{j=1}^\infty \left(b_j b^{w_j}\alpha_j^{-1}\right)^{2\lambda/(1-\lambda)}<\infty.
\end{equation}
We denote the upper bound in \eqref{eq:Alambdabound} by $\overline{A} (\lambda)$. Similarly to what 
is done in \cite[Proof of Theorem 6.4]{KSS12}, we see that Condition \eqref{eq:chooselambda} is satisfied if 
$\lambda\ge \frac{p}{2-p}$.
Again, similarly to \cite[Proof of Theorem 6.4]{KSS12} we see that the latter can be achieved by choosing 
\begin{equation}\label{eq:lambdap}
\lambda_p=\begin{cases}
1/(2-2\delta)\quad\mbox{for some $\delta\in (0,1/2)$} & \mbox{if $p\in (0,2/3]$},\\
p/(2-p) & \mbox{if $p\in (2/3,1)$}.
\end{cases}
\end{equation}
Hence by choosing $\lambda$ equal to $\lambda_p$, we get an efficient bound 
on $C_{\bsgamma,\bsw,\lambda_p} = A_{\lambda_p}^{1 + 1/\lambda_p}$,
as long as the $w_j$ are chosen to guarantee convergence of $\sum_{j=1}^\infty \left(b_j b^{w_j}\right)^{p}$.

\section{Combined error bound}

The derivation in the previous section leads to the following result.
\begin{theorem}	
	Given the PDE in \eqref{eq:PDE} for which we characterized the regularity of the random field by a sequence of 
	$b_j$ with sparsity $p \in (0,1)$ and determined a sequence of $w_j$ such that $\sum_{j=1}^\infty (b_j \, b^{w_j})^p<\infty$, we can construct the generating vector for an $N$-point randomized lattice rule using the reduced CBC algorithm (Algorithm \ref{alg:RedCBCAlgPOD}), at the cost of $\mathcal{O}(\sum_{j=1}^{\min\{s,s^{\ast}\}} (m-w_j+j) \, b^{m-w_j})$ operations, such that, assuming that \eqref{def:kappa}, \eqref{eq:cond_fem} and $\frac{\kappa \, a_{0,\max}}{(1-\overline{\kappa}) \, a_{0,\min}} \sup_{j \ge s+1} b_j < 1$ hold, we obtain an upper bound 
	\begin{equation}\label{eq:combined_error2}
		e^{\text{MSE}}_{N,s,h}(G(u)) \lesssim \left(\sup_{j \ge s+1} b_j \right)^2 + h^4 + \left( \frac2N \right)^{1/\lambda_p},
	\end{equation}
	where the implied constant is independent of $s$, $h$ and $N$.
\end{theorem}
Observe that if the $w_j$ increase sufficiently fast, the construction cost of Algorithm \ref{alg:RedCBCAlgPOD} does not depend anymore on the increasing dimensionality. Further note that the first term on the right-hand side of \eqref{eq:combined_error2} is small if $\sup_{j \ge s+1} b_j$ is small, and, since we assumed that $b_j$ must tend to zero by assumption \eqref{eq:assumpThm1}, we can shrink the first summand by choosing $s$ sufficiently large. By choosing $h$ sufficiently small, and $N$ sufficiently large, we can also make the other two summands in the overall error bound small. 

Note that $\sup_{j\ge s+1} b_j\le \sum_{j\ge s+1} b_j$, and that \eqref{eq:assumpThm1} 
yields $\sum_{j=1}^\infty b_j^p <\infty$, which implies that one can use the machinery developed in \cite{KSS12} 
to obtain a cost analysis similar to \cite[Theorem 8.1]{KSS12}. 
Note, in particular, that it is sufficient to choose $N$ of order 
$\mathcal{O}(\varepsilon^{-\lambda_p /2})$, independently of $s$, to meet an error threshold of $\varepsilon$. 

\section{Derivation of the fast reduced CBC algorithm}
\label{sec:appendix}

Finally in this last section the derivation of the fast reduced CBC algorithm for POD weights in Algorithm~\ref{alg:RedCBCAlgPOD} is given.
For prime $b$ and $m \in \N$ let $N = b^m$.
Consider a generating vector $\widetilde{\bsz} = (b^{w_1} z_1,\ldots, b^{w_d} z_d)$ with $z_j \in \Z_{b^{m-w_j}}^\times$ and integer $0 \le w_j \le m$ for each $j = 1,\ldots,d$.
Furthermore, for an integer $0 \le w' \le m$, the squared worst-case error can be written as
\begin{align*}
  e_{b^m,d}^2(\widetilde{\bsz})
  &=
  \frac{1}{b^m} \sum_{k \in \Z_{b^m}} \sum_{\ell=1}^{d} \sum_{\substack{\setu \subseteq \{1:d\} \\ |\setu| = \ell}}
  \Gamma(\ell)
  \prod_{j \in \setu} \gamma_j \, \omega\!\left(\frac{k \, b^{w_j} z_j \bmod b^m}{b^m}\right)
  \\
  &=
  \frac{1}{b^m} \sum_{k \in \Z_{b^m}} \sum_{\ell=1}^{d} \sum_{\substack{\setu \subseteq \{1:d\} \\ |\setu| = \ell}}
  \Gamma(\ell)
  \prod_{j \in \setu} \gamma_j \, \omega\!\left(\frac{k \, z_j \bmod b^{m-w_j}}{b^{m-w_j}}\right)
  \\
  &= 
  \frac{1}{b^m} \sum_{k' \in \Z_{b^{m-w'}}} \sum_{\ell=1}^{d}
  \underbrace{\sum_{t \in \Z_{b^{w'}}} \sum_{\substack{\setu \subseteq \{1:d\} \\ |\setu| = \ell}}
  \Gamma(\ell)
  \prod_{j \in \setu} \gamma_j \, \omega\!\left(\frac{(k' + t \, b^{m-w'}) \, z_j \bmod b^{m-w_j}}{b^{m-w_j}}\right)}_{=: q_{d,\ell,w'}(k') \text{ for } k' \in \Z_{b^{m-w'}}}
  \\
  &=
  \frac{1}{b^m} \sum_{k' \in \Z_{b^{m-w'}}} \sum_{\ell=1}^{d} q_{d,\ell,w'}(k')
  .
\end{align*}
We note that this holds for any integer $0 \le w' \le m$ and, in particular, for $w=0$ this is the vector being used in the normal fast CBC algorithm.
We now write the error in terms of the previous error, as is standard for CBC algorithms, by splitting the expression into subsets $\setu \subseteq \{1:d\}$ 
for which $d \not\in \setu$ and $d \in \setu$, to obtain
\begin{align*}
  &
  e_{b^m,d}^2(\widetilde{\bsz})
  =
  e^2_{b^m,d-1}(\widetilde{z}_1,\ldots,\widetilde{z}_{d-1}) +
  \\
  &\quad
  \frac{1}{b^m} \sum_{k \in \Z_{b^m}}
  \sum_{\ell=0}^{d-1} \frac{\Gamma(\ell+1)}{\Gamma(\ell)}
  \!\!\sum_{\substack{\setu \subseteq \{1:d-1\} \\ |\setu| = \ell}} \Gamma(\ell)
  \prod_{j \in \setu} \gamma_j \, \omega\!\left(\frac{k \, z_j \bmod b^{m-w_j}}{b^{m-w_j}}\right)
  \,
  \gamma_d \, \omega\!\left(\frac{k \, z_d \bmod b^{m-w_d}}{b^{m-w_d}}\right)
  .
\end{align*}
Since the choice of $z_d \in \Z_{b^{m-w_d}}^\times$ is modulo~$b^{m-w_d}$, we can make a judicious choice for splitting up $k = k' + t\,b^{m-w_d}$ 
for which the effect of dimension~$d$ (for a choice of $z_d$) is then constant for all $t \in \Z_{b^{w_d}}$. We obtain
\begin{align}\label{eq:e2-split}
  e_{b^m,d}^2(\widetilde{\bsz})
  &=
  e^2_{b^m,d-1}
  +
  \frac{1}{b^m} \sum_{k' \in \Z_{b^{m-w_d}}}
  \sum_{\ell=0}^{d-1} \frac{\Gamma(\ell+1)}{\Gamma(\ell)}
  q_{d-1,\ell,w_d}(k')
  \,
  \gamma_d \, \omega\!\left(\frac{k' \, z_d \bmod b^{m-w_d}}{b^{m-w_d}}\right)
  .
\end{align}
Then we observe that for all $0 \le w_{d-1} \le w_d \le m$, with $k' \in \Z_{b^{m-w_d}}$, writing $t = t' + t'' \, b^{w_d-w_{d-1}} \in \Z_{b^{w_d}}$ 
with $t' \in \Z_{b^{w_d-w_{d-1}}}$ and $t'' \in \Z_{b^{w_{d-1}}}$, leads to
\begin{align*}
  q_{d-1,\ell,w_d}(k')
  &=
  \sum_{t' \in \Z_{b^{w_d-w_{d-1}}}} \sum_{t'' \in \Z_{b^{w_{d-1}}}}
  \\
  &\qquad\quad
  \sum_{\substack{\setu \subseteq \{1:d-1\} \\ |\setu| = \ell}}
  \Gamma(\ell)
  \prod_{j \in \setu} \gamma_j \, \omega\!\left(\frac{(k' + (t' + t'' \, b^{w_d-w_{d-1}}) \, b^{m-w_d}) \, z_j \bmod b^{m-w_j}}{b^{m-w_j}}\right)
  \\
  &=
  \sum_{t' \in \Z_{b^{w_d-w_{d-1}}}} q_{d-1,\ell,w_{d-1}}(k' + t' \, b^{m-w_d})
  ,
\end{align*}
where $k'' = k' + t' \, b^{m-w_d} \in \Z_{b^{m-w_{d-1}}}$ as required for $q_{d-1,\ell,w_{d-1}}(k'')$.
Note that this is the property of the ``fold and sum'' operator as introduced in~\eqref{eq:fold-and-sum} and mentioned there.
Using matrix-vector notation, we rewrite the expression in \eqref{eq:e2-split} for all $z_d \in \Z_{b^{m-w_d}}^\times$ as
\begin{align*}
  \bse_{b^m,d}^2
  &=
  e_{b^m,d-1}^2
  +
  \frac{\gamma_d}{b^m} \, \Omega_{b^{m-w_d}}
	\left( \sum_{\ell=0}^{d-1} \frac{\Gamma(\ell+1)}{\Gamma(\ell)}
	\left[ P_{w_d,w_{d-1}}^m \, \bsq_{d-1,\ell,w_{d-1}} \right] \right) ,
\end{align*}
where $\bse_{b^m,d}^2 \in \R^{\varphi(b^{m-w_d})}$ is the vector with components $e_{b^m,d}^2(b^{w_1} z_1, \ldots, b^{w_d} z_d)$ for all $z_d \in \Z_{b^{m-w_d}}^\times$.
After $z_d$ has been selected we can calculate (for $\ell=1,\ldots,d$)
\begin{align*}
	\bsq_{d,\ell,w_d} 
	= 
	\left[ P_{w_d,w_{d-1}}^m \, \bsq_{d-1,\ell,w_{d-1}} \right]
	+
	\frac{\Gamma(\ell)}{\Gamma(\ell-1)} \gamma_d \, 
	\Omega_{b^{m-w_d}}(z_d,:) \, {.*} \, 
	\left[ P_{w_d,w_{d-1}}^m \, \bsq_{d-1,\ell-1,w_{d-1}} \right] .
\end{align*}
In Algorithm~\ref{alg:RedCBCAlgPOD} the vectors $\bsq_{j,\ell,w_j}$ are denoted by just $\bsq_{j,\ell}$.


\begin{small}
	\noindent\textbf{Authors' addresses:}\\
	
	\noindent Adrian Ebert\\
	Department of Computer Science\\
	KU Leuven\\
	Celestijnenlaan 200A, 3001 Leuven, Belgium.\\
	\texttt{adrian.ebert@cs.kuleuven.be}
	
	\medskip
	
	\noindent Peter Kritzer\\
	Johann Radon Institute for Computational and Applied Mathematics (RICAM)\\
	Austrian Academy of Sciences\\
	Altenbergerstr. 69, 4040 Linz, Austria.\\
	\texttt{peter.kritzer@oeaw.ac.at}
	
	\medskip
	
	\noindent Dirk Nuyens\\
	Department of Computer Science\\
	KU Leuven\\
	Celestijnenlaan 200A, 3001 Leuven, Belgium.\\
	\texttt{dirk.nuyens@cs.kuleuven.be}
	
\end{small}

\end{document}